\documentclass[11pt,a4paper,twoside]{article}
\usepackage{etoolbox}  
\usepackage{forloop}    
\usepackage{amsthm,amsfonts,amsmath,amscd,amssymb}
\usepackage{latexsym}
\usepackage{euscript}
\usepackage{enumitem}
\usepackage{graphicx}
\usepackage{tikz}
\usepackage{caption}
\usepackage{subcaption}
\usepackage{cite}
\usepackage[utf8]{inputenc}
\usepackage[english]{babel}
\usepackage{xcolor}
\definecolor{darkgreen}{HTML}{00A000}
\definecolor{darkblue}{HTML}{0000A0}
\definecolor{darkred}{HTML}{D00000}
\usepackage[colorlinks=true]{hyperref}
\hypersetup{linkcolor=darkred, urlcolor=darkblue, citecolor=darkgreen}
\usepackage{jmpag}

\begin{document}


\title[Gap control by singular Schr\"odinger operators]{Gap control by singular Schr\"odinger operators in a periodically structured metamaterial}



\author{Pavel Exner}
\address{Nuclear Physics Institute, Academy of Sciences of the Czech Republic,
Hlavn\'{i} 130, \v{R}e\v{z} near Prague, 25068, Czech Republic\\
Doppler Institute, Czech Technical University, B\v{r}ehov\'{a} 7, Prague, 11519, Czech Republic}
\email{exner@ujf.cas.cz}

\author{Andrii Khrabustovskyi}
\address{Institute of Applied Mathematics, Graz
Institute of Technology, Steyrergasse 30, Graz, 8010, Austria}
\email{khrabustovskyi@math.tugraz.at}


\BeginPaper 


\newcommand{\cupl}{\bigcup\limits}
\newcommand{\supp}{\mathrm{supp}}
\newcommand{\suml}{\sum\limits}
\newcommand{\intl}{\int\limits}
\newcommand{\liml}{\lim\limits}
\newcommand{\maxl}{\max\limits}
\newcommand{\minl}{\min\limits}
\newcommand{\e}{^\varepsilon}
\newcommand{\eps}{{\varepsilon}}
\newcommand{\LL}{\mathsf{L}^2}
\newcommand{\W}{\mathsf{H}^1}
\newcommand{\ds}{\displaystyle}
\newcommand{\dd}{\,\mathrm{d}}
\newcommand{\Hl}{\mathcal{H}}
\newcommand{\h}{\mathfrak{h}}
\newcommand{\HH}{\widehat{\mathcal{H}}}
\newcommand{\hh}{\widehat{\mathfrak{h}}}
\newcommand{\R}{\mathbb{R}}
\newcommand{\N}{\mathbb{N}}
\newcommand{\C}{\mathbb{C}}
\newcommand{\Z}{\mathbb{Z}}
\newcommand{\dom}{\mathrm{dom}}
\newcommand{\n}{\mathbf{n}}
\newcommand{\m}{\{1,\dots,m\}}
\newcommand{\Id}{\mathrm{I}}
\newcommand{\uu}{{\mathbf{u}}}
\newcommand{\restr}{\hspace{-1pt}\restriction}
\newcommand{\ext}{^{\rm ext}}
\newcommand{\inn}{^{\rm int}}
\newcommand{\gaij}{_{_{\Gamma_{ij}^\eps}}}
\newcommand{\gaj}{_{_{\Gamma_j}}}



\begin{center}\small\it\medskip
This paper is dedicated to Volodymyr Olexandrovych Marchenko\\
on the occasion of his jubilee\bigskip
\end{center}

\begin{abstract}
We consider a family $\{\mathcal{H}^\varepsilon\}_{\varepsilon>0}$ of $\varepsilon\mathbb{Z}^n$-periodic Schr\"odinger operators with $\delta'$-interactions supported on a lattice of closed compact surfaces; within a minimal period cell one has $m\in\N$ surfaces. We show that in the limit when $\varepsilon\to 0$ and the interactions strengths are appropriately scaled, $\mathcal{H}^\varepsilon$ has at most $m$ gaps within finite intervals, and moreover, the limiting behavior of the first $m$ gaps can be completely controlled through a suitable choice of those surfaces and of the interactions strengths.

\key{periodic Schr\"{o}dinger operators, $\delta'$ interaction, spectral gaps, eigenvalue asymptotics.}

\end{abstract}


\section{Introduction\label{sec1}}

Spectral analysis of operators with periodic coefficients is a traditional topic in mathematical physics. It received a new strong motivation recently coming from the advances in investigation of \emph{metamaterials} of various sorts. One of the central questions concerns the structure of spectral gaps in view of their importance for conductivity properties of such substances, in particular, the possibility of engineering the gap structure by choosing a properly devised material texture. In the present paper we investigate this problem for a class of such operators; we are going to show that using a suitable lattice of `traps' arranged periodically in combination with a scaling transformation that makes these traps smaller and  weaker one can approximate any prescribed finite family of spectral gaps. Let us recall in this connection that similar ideas can also appear in a different context, for instance, concerning the gap creation by `decoration' of quantum graphs \cite{BK15}, \cite[Sec.~5.1]{BK}.

The idea to employ $\delta'$ traps was first used in our recent paper \cite{EK15} where we demonstrated that it can provide an approximation to the first spectral gap in the particular case of operators used to model \emph{nanowires} regarding them as electron waveguides. In the said paper we focused our attention to guides with Neumann boundary characteristic for metallic nanowires, and we also supposed that the scaling makes the duct thin. Here we extend this result in two directions. First of all, we suppose that the family of traps is periodic in more than one direction, and secondly, we manage to get an approximation with any finite number of prescribed gaps. What is equally important, however, not only the present result is more general but also the method we employ is different from that used in \cite{EK15} where the argument was based on eigenvalue convergence for the elements of the fibre decomposition by constructing approximations for the  eigenfunctions.

In contrast, in the current paper we identify the limiting operators using Simon's results for a monotonic sequence of forms \cite{S78}. The convergence of the eigenvalues is then proven using a (slightly modified) lemma from \cite{EP05}. This allows us not only to prove the said convergence of eigenvalues, but also to estimate its rate. Location of the spectral gaps can be then controlled by a suitable choice of the interaction `strength' and the trap shapes, that is, surfaces supporting these interactions, following a result from\cite{Kh14}. In the next section we describe the problem properly and state the main result, Section~\ref{sec3} is then devoted to its proof; in the appendix we recall the above indicated lemma.


\section{\label{sec2} Setting of the problem and main result}

Let $m\in\N$ and let $\{\Omega_j\}_{j=1}^m$ be a family of  simply connected Lipschitz domains in $\R^n$, $n\in\N\setminus\{1\}$. We assume that
$$
\overline{\Omega_{j}}\cap \overline{\Omega_{j'}}=\emptyset\text{ as }j\not=j'\text{\quad and \quad }\overline{\cup_{j=1}^m \Omega_j}\subset Y:=(0,1)^n.
$$
Also, we set
$$\Omega_0:=Y\setminus\cupl_{j=1}^m\overline{\Omega_j}.$$
In what follows $\eps>0$ will be a small parameter. For $i\in\Z^n$ and $j\in\m$ we set
$$
\Gamma_{ij}\e:=\eps(\partial \Omega_j + i).
$$

Next we describe the family of operators $\Hl\e$ which will be the main object of our interest in this paper. We denote
$$\Gamma\e= \cupl_{i\in\Z^n}\, \cupl_{j=1}^m \Gamma_{ij}\e
$$
and introduce the sesquilinear form $\h\e$ in the Hilbert space $\LL(\R^n)$ via
\begin{multline}\label{h}
\h\e[u,v]:=\int_{\R^n\setminus \Gamma\e}\nabla u\cdot\nabla\bar{v} \dd x\\+ \eps \sum_{i\in\Z}\sum_{j=1}^m q_j\int_{\Gamma_{ij}\e}(u\restr\gaij\ext-u\restr\gaij\inn)\overline{(v\restr\gaij\ext-v\restr\gaij\inn)}\,\dd s,\quad q_j>0,
\end{multline}
with the form domain $\dom(\h\e)=\W (\R^n\setminus \Gamma\e)$. Here $f\restr\gaij\ext$ (respectively, $f\restr\gaij\inn$) stands for the trace of the function $f$ taken from the exterior (respectively, interior) side of ${\Gamma_{ij}\e}$; $\dd s$ is the `area' measure on $\Gamma_{ij}\e$.

\begin{remark}
{\rm From the viewpoint of the physical motivation mentioned in the introduction the cases $n=2,3$ are important, however, there is no problem in stating and proving the result for any dimension; what matters is that the codimension of the interaction support is one. In general the trap lattice may have different periods in different dimensions but using suitable scaling transformations one can reduce such situations to the case considered here.}
\end{remark}

The definition of $\h\e[u,v]$ makes sense: the second sum in \eqref{h} is finite as one can check applying the standard trace inequalities within each period cell. Furthermore, it is straightforward to check that the form $\h\e[u,v]$ is densely defined, closed, and positive. Then by the first representation theorem \cite[Chapter 6, Theorem 2.1]{K66} there is a the unique self-adjoint and positive operator associated with the form $\h\e$, which we denote as $\Hl\e$,
\begin{gather*}
(\Hl\e u,v)_{\LL(\R^n)}= \h\e[u,v],\quad\forall u\in
\dom(\Hl\e),\ \forall  v\in \dom(\h\e).
\end{gather*}

Let $u\in\dom(\Hl\e)\cap \mathsf{C}^2(\R^n\setminus \Gamma\e)$. Integrating by parts one can easily show that
$$
(\Hl\e u) (x)=-\Delta u(x)\quad\text{at }\;x\in\R^n\setminus \Gamma\e,
$$
while on   $\Gamma_{ij}\e$ one has
the following interface matching conditions,
\begin{gather*}
\left({\partial_\n u }\right)\restr\gaij\ext=\left({\partial_\n u }\right)\restr\gaij\inn=\eps q_j(u\restr\gaij\ext-u\restr\gaij\inn),
\end{gather*}
where $\partial_\n$ is the derivatve along the outward-pointing unit normal to $\Gamma_{ij}\e$. This supports our interpretation of $\Hl\e$ as the Hamiltonians describing an lattice of periodically spaced obstacles, or `traps' in the form of given by a $\delta'$ interaction supported by $\Gamma_{ij}\e$; the interaction becomes `weak' as $\eps\to 0$. For more details on Schr\"odinger operators in $\R^n$ with  $\delta'$ interactions supported by hypersurfaces we refer to \cite{BEL14,BLL13}.

We denote by $\sigma(\Hl\e)$ the spectrum of $\Hl\e$. Due to the Floquet-Bloch theory $\sigma(\Hl\e)$ is a locally finite union of compact intervals called \textit{bands}. In general the bands may touch each other or even overlap. The non-empty bounded open interval $(A,\b)\subset\R$ is called a \textit{gap} in the spectrum of $\Hl\e$ if
$$
(A,B)\cap\sigma(\Hl\e)=\emptyset,\quad A,B\in\sigma(\Hl\e).
$$

First we give a simple estimate from above to the number of gaps.

\begin{proposition}\label{prop1}
The spectrum $\sigma(\Hl\e)$ has at most $m$ gaps within the interval $[0,\Lambda\eps^{-2}]$ with some constant $\Lambda>0$ depending on the set $\Omega_0$ only.
\end{proposition}
The proof of this proposition is simple, but we postpone it to Section~\ref{sec3}, cf. Corollary~\ref{coro1}, since we need to do some preliminary work first. The constant $\Lambda$ is given by \eqref{Lambda}.

\smallskip

Our main goal  is to detect gaps in the spectrum of $\Hl\e$ within the interval $[0,\Lambda\eps^{-2}]$ and to describe their asymptotic behavior as $\eps\to 0$. To state the result we have to introduce some notations.

In what follows we denote by $C$, $C_1$, etc.\ generic constants being independent of $\eps$ and of functions appearing in the estimates and equalities where these constants occur, however, they may depend on $n$, $\Omega_j$ and $q_j$.

For $j\in\m$ we set
\begin{gather*}
A_j:={q_j |\partial \Omega_j|\over |\Omega_j|},
\end{gather*}
where the symbol $|\cdot|$ serves both for the volume of domain in $\R^n$ and for the `area' of $(n-1)$-dimensional surface in $\R^n$. We assume that the domains $\Omega_j$ and the numbers $q_j$ are chosen in such a way that
\begin{gather}\label{alpha-cond}
A_j<A_{j+1},\ j\in\{1,\dots,m-1\}.
\end{gather}
Furthermore, we consider the following rational function,
\begin{gather}\label{mu_eq}
F(\lambda):=1+\sum_{j=1}^m{A_j |\Omega_j|\over |\Omega_0|(A_j-\lambda)}.
\end{gather}
It is easy to show that $F(\lambda)$ has exactly $m$ roots, those are real and interlace with $A_j$ provided \eqref{alpha-cond} holds. We denote
them $B_j$, $j\in\m$ assuming that they are renumbered in the ascending order,
\begin{gather}\label{inter}
A_j<B_j<A_{j+1},\quad j\in\{1,\dots,m-1\},\quad
A_m<B_m<\infty.
\end{gather}

Now we are in position to formulate the main results of this work.

\begin{theorem}\label{th1}
The spectrum of $\Hl\e$ has the following form within the interval $[0,\Lambda\eps^{-2}]$,
\begin{gather*}
\sigma(\Hl\e)\cap [0,\Lambda\eps^{-2}]=[0,\Lambda\eps^{-2}]\setminus \left(\bigcup\limits_{j=1}^m (A_j\e,B_j\e)\right).
\end{gather*}
The endpoints of the intervals $(A_j\e,B_j\e)$ satisfy
\begin{gather*}
A_j\e\in [A_j-C\eps,\,A_j],\quad
B_j\e\in [B_j-C\eps,\,B_j],
\end{gather*}
provided $\eps$ is small enough.
\end{theorem}
\begin{remark} \label{r:small}
{\rm In the above theorem `provided $\eps$ is small enough' means $\eps<\eps_0$ for some $\eps_0$ which depends in general on $q_j$ and $\Omega_j$. It will be apparent from the proof, cf.~Lemma~\ref{lemmaN2}, that $\eps_0$ can be given explicitly, but the formula looks rather cumbersome, in particular, it depends on the constants appearing in the Poincar\'e and trace inequalities for $\Omega_j$.}
\end{remark}

Using a lemma from \cite{Kh14} \emph{one can choose the domains $\Omega_j$ and the numbers $q_j$ in such a way that the limiting intervals $(A_j,B_j)$ coincide with predefined segments}. Indeed, let us defined the map
$$
\mathcal{L}:\dom(\mathcal{L})\subset\R^{2m}\to\R^{2m},\quad (a_1,\dots,a_m,b_1,\dots,b_m)\overset{\mathcal{L}}\mapsto (A_1,\dots,A_m,B_1,\dots,B_m)
$$
with the domain
$$
\dom(\mathcal{L})=\bigg\{(a_1,\dots,a_m,b_1,\dots,b_m)\in\R^{2m}:\ a_j>0,\ b_j>0,\ \sum_{j=1}^m b_j<1,\ {a_j\over b_j}<{a_{j+1}\over b_{j+1}}\bigg\}
$$
acting as follows:
$\ds A_j={a_j\over b_j}$,  $B_j$ are the roots of the function
$$
\ds 1+\sum_{j=1}^m{A_j b_j\over b_0(A_j-\lambda)},\;\text{ where }\ds b_0:=1-\sum_{j=1}^m b_j,
$$
renumbered according to \eqref{inter}. The indicated result \cite[Lemma~2.1]{Kh14} then reads as follows:

\begin{lemma}\label{lemma1}
	${\mathcal{L}}$ maps $\mathrm{dom}(\mathcal{L})$ \textbf{onto} the set
	of $(A_1,\dots,A_m,B_1,\dots B_m)\in\R^{2m}$ satisfying \eqref{inter}. Moreover $\mathcal{L}$ is
	one-to-one and the inverse map $\mathcal{L}^{-1}$ is given by the
	following formul{\ae},
	\begin{gather}\label{inverse}
	a_j=
	A_j\ds{ \rho_j\over 1+\suml_{i=1}^m
		\rho_i},\quad b_j={\rho_j\over 1+\suml_{i=1}^m
		\rho_i},
	\end{gather}
	where
	\begin{gather*}
	\rho_j={B_j-A_j\over A_j}\prod\limits_{i={1,\dots,m}|i\not=
		j}\ds\left({B_i-A_j\over A_i-A_j}\right).
	\end{gather*}
\end{lemma}

Now it is clear how to choose the sought $\Omega_j$ and $q_j$, cf. the statement following Remark~\ref{r:small}.  Specifically, assume that intervals $(A_j,B_j)$ satisfying \eqref{inter} are given. We define for them the numbers $a_j$, $b_j$ by formul{\ae} \eqref{inverse}, and then we choose domains $\Omega_j$, $j\in\m$, in such a way that $|\Omega_j|=b_j$. Obviously, this can be always done since $b_j>0$ and $\sum_{j=1}^m b_j<1\,$; recall that the closures of $\Omega_j$ must be pairwise disjoint by assumption and belong to the unit cube. Needless to say, such a choice is not unique. Finally, with these $\Omega_j$ we define the numbers $q_j$ by $q_j={A_j |\Omega_j|\over |\partial \Omega_j|}$.

\section{Proof of the results\label{sec3}}

\subsection{Preliminaries}

We introduce the sets
\begin{itemize}
\setlength{\itemsep}{5pt}
\item $\Gamma_j=\partial \Omega_j$, where $j\in\m$
\item $\Gamma_{ij}=\partial \Omega_j + i$, where $i\in\Z^n$, $j\in\m$
\item $\Gamma=\cup_{i\in\Z^n}\cup_{j\in \m}\Gamma_{ij} $.
\end{itemize}

The operator $\Hl\e$ is by construction $\Z^n$-periodic with the period cell $\eps Y$. It is convenient to perform a change of coordinates $x = \eps y$ (from the old coordinates~$x$ to the new coordinates~$y$) that would allow us to work with an $\eps$-independent period cell. More precisely, we introduce the sesquilinear form $\hh\e$ in the Hilbert space $\LL(\R^n)$ via
\begin{multline*}
\hh\e[u,v]:={1\over\eps^2}\int_{\R^n\setminus \Gamma}\nabla u\cdot\nabla\bar{v} \dd x\\+ \sum_{i\in\Z}\sum_{j=1}^m q_j\int_{\Gamma_{ij}}(u\restr_{_{\Gamma_{ij}}}\ext-u\restr_{_{\Gamma_{ij}}}\inn)\overline{(v\restr_{_{\Gamma_{ij}}}\ext-v\restr_{_{\Gamma_{ij}}}\inn)}\,\dd s,\quad q_j>0,
\end{multline*}
with the form domain $\dom(\hh\e)=\W (\R^n\setminus \Gamma)$.  Finally, by $\HH\e$ we denote the unique self-adjoint and positive operator  associated with the form $\hh\e$. It is easy to see that
$$
\sigma(\HH\e)=\sigma(\Hl\e).
$$
Moreover, the operator $\HH\e$ is periodic with respect to the $\eps$-independent period cell $Y$.

The Floquet-Bloch theory -- see, e.g., \cite{BHPW11,Ku93,RS78} -- establishes a relationship between $\sigma(\HH\e)$ and the spectra of certain operators on $Y$. Let $\phi=(\phi_1,\dots,\phi_n)\in [0,2\pi)^n$, the dual cell to $Y$. We introduce the space $\W_\phi(Y\setminus \cup_{j=1}^m \Gamma_j)$, which consists of functions from $\W(Y\setminus \cup_{j=1}^m \Gamma_j)$ satisfying the following conditions at the opposite faces of $\partial Y$, usually referred to as \emph{quasi-periodic boundary conditions},
\begin{multline}\label{quasi}
\forall k\in\{1,\dots,n\}:\  u(x+ e_k)=\exp(i\phi_k) u(x)\quad \text{for}\;x=\underset{^{\overset{\qquad\quad\uparrow}{\qquad\quad
k\text{-th place}}\qquad }}{(x_1,x_2,\dots,0,\dots,x_n)},
\end{multline}
where $e_k={(0,0,\dots,1,\dots,0)}$.

In the space $\LL(Y)$ we introduce the sesquilinear form $\hh_\phi\e$ defined by
\begin{multline}\label{h-phi}
\hh_\phi\e[u,v]:={1\over\eps^2}\int_{Y \setminus \cup_{j=1}^m \Gamma_j}\nabla u\cdot\nabla\bar{v} \dd x\\+\suml_{j=1}^m q_j\int_{\Gamma_j}(u\restr\gaj\ext-u\restr\gaj\inn)\overline{(v\restr\gaj\ext-v\restr\gaj\inn)}\dd s
\end{multline}
with the domain $\W_\phi(Y \setminus \cup_{j=1}^m \Gamma_j)$. We denote by $\HH_\phi^{\eps}$ the associated self-adjoint and positive operator. Its
domain consists of functions $u\in \mathsf{H}^2(Y \setminus \cup_{j=1}^m \Gamma_j)\cap \W_\phi(Y\setminus \cup_{j=1}^m \Gamma_j)$ satisfying also
\begin{multline}\label{quasi+}
\forall k\in\{1,\dots,n\}:\\{\partial u\over \partial x_k}(x+ e_k)=\exp(i\phi_k) {\partial u\over \partial x_k}(x)\quad \text{
for}\;x=\underset{^{\overset{\qquad\quad\uparrow}{\qquad\quad
k\text{-th place}}\qquad }}{(x_1,x_2,\dots,0,\dots,x_n)}
\end{multline}
and the following $\delta'$ interface matching conditions on $\Gamma_j$,
\begin{gather*}
\left({\partial_\n u }\right)\restr\gaj\ext=\left({\partial_\n u }\right)\restr\gaj\inn=  \eps^2 q_j(u\restr\gaj\ext-u\restr\gaj\inn),
\end{gather*}
where $\partial_\n$ is the derivative along the outward-pointing unit normal to $\Gamma_j$. The operator $\HH\e_\phi$ acts as
$$
(\HH\e_\phi u)\restriction_{\Omega_j}=-{1\over\eps^2}(\Delta u)\restriction_{\Omega_j},\quad j\in\{0,\dots,m\}.
$$
The spectrum of $\HH_\phi\e$ is purely discrete. We denote by $\big\{\lambda_{k,\phi}\e\big\}_{k\in\mathbb{N}}$ the sequence of its eigenvalues  arranged in the ascending order and repeated according to their multiplicities.

According to the Floquet-Bloch theory we have
\begin{gather}\label{repres1}
\ds\sigma(\HH\e)=\cupl_{k=1}^\infty \cupl_{\phi\in [0,2\pi)^n}
\big\{\lambda_{k,\phi}\e\big\},
\end{gather}
and moreover, for any fixed $k\in\mathbb{N}$ the set $\cup_{\phi\in [0,2\pi)^n} \big\{\lambda_{k,\phi}\e\big\}$ is a compact interval, conventionally referred to as the $k$th spectral band.

Along with the operators $\HH\e_\phi$ we also introduce the operators $\HH\e_N$ and $\HH\e_D$, which differ from $\HH\e_\phi$ only by the boundary conditions at $\partial Y$: instead of the quasi-periodic conditions one imposes here the Neumann and the Dirichlet ones, respectively.
More precisely, we introduce in $\LL(Y)$ the sesquilinear forms $\hh\e_N$ and $\hh\e_D$ with the domains
$$
\dom(\hh\e_N)=\W (Y \setminus \cup_{j=1}^m \Gamma_j)\text{\quad and\quad }
\dom(\hh\e_D)=\left\{u\in \W(Y \setminus \cup_{j=1}^m \Gamma_j):\ u\restriction_{Y}=0\right\}
$$
and the action specified by \eqref{h-phi}; then $\HH\e_N$ and $\HH\e_D$ are the operators associated with these forms. The spectra of these operators  are purely discrete. We denote by $\big\{\lambda_{k,N}\e\big\}_{k\in\mathbb{N}}$ (respectively, $\big\{\lambda_{k,D}\e\big\}_{k\in\mathbb{N}}$) the sequence of eigenvalues of $\HH_N^{\eps}$ (respectively, of $\HH_D^{\eps}$) arranged in the ascending order and repeated according to their multiplicities. Since
$$
\forall\phi\in [0,2\pi)^n:\quad \dom(\hh\e_N)\supset\dom(\hh\e_\phi)\supset\dom(\hh\e_D),
$$
using the min-max principle \cite[Sec.~XIII.1]{RS78} we obtain
\begin{gather}\label{enclosure}
\forall k\in \mathbb{N},\ \forall\phi\in [0,2\pi)^n:\quad
\lambda_{k,N}\e \leq \lambda_{k,\phi}\e \leq
\lambda_{k,D}\e.
\end{gather}
For a fixed $\phi\in [0,2\pi)^n$ we denote by $\Delta_{N,\phi}(\Omega_0)$ the  Laplace operator on $\Omega_0$ subject to the Neumann conditions on $\cup_{j=1}^m\partial \Omega_j$ and conditions \eqref{quasi}, \eqref{quasi+} on $\partial Y$.

\begin{lemma}
For each $\phi\in [0,2\pi)^n$ one has
$$
{1\over\eps^2}\Lambda_\phi\leq \lambda_{m+1,\phi}\e,
$$
where $\Lambda_\phi$ is the smallest eigenvalue of the operator $-\Delta_{N,\phi}(\Omega_0)$.
\end{lemma}
\begin{proof}
We consider the \emph{decoupled} operator
$$
\HH_{\phi,{\rm dec}}\e= \left(-{1\over\eps^2}\Delta_{N,\phi}(\Omega_0)\right)\oplus\left(\oplus_{j=1}^m \left(-{1\over\eps^2}\Delta_N(\Omega_j)\right)\right),
$$
where $\Delta_N(\Omega_j)$ is the Neumann Laplacian on $\Omega_j$, $j=1,\dots,m$. Since $q_j>0$ we get
\begin{gather}\label{h-h}
\hh_{\phi,{\rm dec}}\e\leq \hh_{\phi}\e,
\end{gather}
where $\hh_{\phi,{\rm dec}}\e$ is the form associated with $\HH_{\phi,{\rm dec}}\e$. Using the min-max principle, we conclude from \eqref{h-h} that the $k$th eigenvalue of $\HH_{\phi,{\rm dec}}\e$ is smaller or equal than the $k$th eigenvalue of $\HH_{\phi}\e$ for any $k\in\N$. It is clear that the first $m$ eigenvalues of $\HH_{\phi,{\rm dec}}\e$ are equal to zero, while the $(m+1)$th one equals $\eps^{-2}\Lambda_\phi$, whence we obtain the desired result.
\end{proof}

Now we set
\begin{gather}\label{Lambda}
\Lambda:=\max_{\phi\in[0,2\pi)^n}\Lambda_\phi.
\end{gather}
It is easy to see that $\Lambda<\infty$. Indeed, due to the min-max principle, $\Lambda\leq\Lambda_D$, where $\Lambda_D$ is the smallest eigenvalue of the Laplace operator in $\Omega_0$ subject to the Neumann conditions at $\cup_{j=1}^m\partial\Omega_j$ and the Dirichlet conditions at $\partial Y$.
Note, that $\Lambda\not=\Lambda_D$ in general.

From the above lemma and \eqref{repres1} we immediately obtain the following corollary justifying the claim of ~Proposition~\ref{prop1}:
\begin{corollary}\label{coro1}
$\sigma(\HH\e)$ (hence also $\sigma(\Hl\e)$) has at most $m$ gaps on the interval $[0,\Lambda\eps^2]$.
\end{corollary}

Now we are able to proceed to the proof of our main result. First we sketch our strategy.

\subsection{Sketch of the proof}

We distinguish two points of the dual lattice cell, usually referred to as \emph{Brillouin zone}, denoting
$$
\phi_0=(0,0,\dots,0),\quad\phi_\pi=(\pi,\pi,\dots,\pi),
$$
In view of \eqref{repres1}-\eqref{enclosure} the left edge of the $k$th spectral band of $\HH\e$ is located between $\lambda_{k,N}\e$ and $\lambda_{k,\phi_0}$, while the right edge between $\lambda_{k,\phi_\pi}\e$ and $\lambda_{k,D}$. Clearly, $\lambda_{1,N}\e=\lambda_{1,\phi_0}\e=0$ holds. Our goal is to prove that
$$
\begin{array}{ll}
\lim_{\eps\to 0}\lambda_{k,N}\e=\lim_{\eps\to 0}\lambda_{k,\phi_0}\e=B_{k-1},&k=2,\dots,m+1,\\[2mm]
\lim_{\eps\to 0}\lambda_{k,D}\e=\lim_{\eps\to 0}\lambda_{k,\phi_\pi}\e=A_{k},&k=1,\dots,m,
\end{array}
$$
and moreover, that the rate of this convergence is of order $C\eps$. These results taken together constitute the claim of Theorem~\ref{th1}.

Let us start from the Neumann eigenvalues. The idea is to find a limit operator $\HH_N$ the eigenvalues of which will approach $\lambda_{k,N}\e$ as $\eps\to 0$. It is not difficult to guess -- using, e.g., Simon's results \cite{S78} about monotonic sequences of forms -- how the `limit' operator should looks like: it is associated with the form
\begin{gather}
\label{simon}
\begin{array}{l}
\dom(\hh_N)=\left\{u\in \cap_{\eps>0}\,\dom(\hh_N\e):\ \sup_{\eps>0}\,\hh_N\e[u,u]<\infty \right\},\\[2mm]  \hh_N[u,v]=\lim_{\eps\to 0}\,\hh_N\e[u,v].
\end{array}
\end{gather}
Evidently, $\dom(\hh_N)$ consists of functions being constant on each $\Omega_j$ and the value of the form on functions, with the abuse of notation written as $u=(u_0,\dots,u_m)\in\C^{m+1}$, is given by $\suml_{j=1}^m q_j |\Gamma_j| |u_j-u_0|^2$. Moreover, it turns out that the eigenvalues of $\HH_N$ are $0,\, B_1,\,\dots,B_m$, with the reference to a result obtained in \cite{BK15}.

The limit operator for $\HH\e_{\phi_0}$ is again $\HH_N$, since function being constant on $\Omega$ satisfy $\phi_0$-periodic boundary conditions, and consequently, \eqref{simon} leads to the same operator.

The limit operator for $\HH\e_{D}$ is associated with the form $\hh_D$ defined by \eqref{simon} except that $\hh_N\e$ is replaced by $\hh_D\e$. Since the only constant satisfying the Dirichlet boundary conditions is zero, we conclude that $\dom(\hh_D)=\C^m$ and the action of this form on $u=(u_1,\dots,u_m)\in\C^{m}$ is $\suml_{j=1}^m q_j |\Gamma_j| |u_j |^2$. The eigenvalues of $\HH_D$ are thus  $A_1,\,\dots,A_m$.

Finally, the limit operator for $\HH\e_{\phi_\pi}$ is  $\HH_D$, since functions being constant on $\Omega$ can satisfy $\phi_\pi$-periodic boundary conditions \emph{iff} that constant is zero.

In the subsequent sections we will implement this strategy. Our asymptotic analysis will be based on a (slighty modified) result from \cite{EP05}  which for the reader's convenience is presented in the Appendix.

\subsection{Asymptotic behavior of $\lambda_{k,N}\e$ and $\lambda_{k,\phi_0}\e$}

In the following we will work with the space $\C^{m+1}$ denoting its elements by bold letters, $\uu$, $\mathbf{v},\dots\,$. Their entries will be enumerated starting from zero,
$$
\uu\in \C^{m+1}\quad \Rightarrow\ \uu=(u_0,\dots,u_m)\text{ with }u_j\in\C.
$$
Let $\C^{m+1}_\Omega$ be the same space $\C^{m+1}$, but equipped with the weighted scalar product,
\begin{gather}
\label{scalar-product}
(\uu,\mathbf{v})_{\C^{m+1}_\Omega}=\suml_{j=0}^m u_j \overline{v_j} |\Omega_j|,
\end{gather}
In this space we introduce the sesquilinear form
$$
\hh_N:\: \hh_N[\uu,\mathbf{v}]=\suml_{j=1}^m q_j |\Gamma_j| (u_j-u_0)\overline{(v_j-v_0)}
$$
with $\dom(\hh_N)=\C^{m+1}_\Omega$. Let $\HH_N$ be the operator in $\C^{m+1}_\Omega$ associated with this form. It is obvious that $\HH_N$ can be represented by the $(n+1)\times(n+1)$ matrix, symmetric with respect to the scalar product \eqref{scalar-product}),
\begin{gather}\label{HHN}
\HH_N=\left(\begin{matrix}\suml_{j=1}^m q_j |\Gamma_j||\Omega_0|^{-1}&-q_1|\Gamma_1| |\Omega_0|^{-1}&-q_2|\Gamma_2||\Omega_0|^{-1}&\dots&-q_m |\Gamma_m| |\Omega_0|^{-1}\\-q_1|\Gamma_1| |\Omega_1|^{-1}&q_1|\Gamma_1| |\Omega_1|^{-1}&0&\dots&0\\
-q_2|\Gamma_2| |\Omega_2|^{-1}&0&q_2|\Gamma_2| |\Omega_2|^{-1}&\dots&0\\
\vdots&\vdots&\vdots&\ddots&\vdots\\
-q_m |\Gamma_m||\Omega_m|^{-1}&0&0&\dots&q_m|\Gamma_m| |\Omega_m|^{-1}
\end{matrix}\right).
\end{gather}
We denote by $\lambda_{1,N}\leq\lambda_{2,N}\leq\dots\leq \lambda_{m+1,N}$ the eigenvalues of $\HH_N$.

\begin{lemma}
	\label{lemmaN1}
For any $k\in\{1,\dots,m+1\}$ one has
$$
\lambda_{k,N}\e\leq  \lambda_{k,N}.
$$
\end{lemma}
\begin{proof}
By the min-max principle we have
\begin{gather}\label{minmax}
\lambda_{k,N}\e=\min_{V\in \mathfrak{V}[k]}\max_{u\in V\setminus\{0\}}{\hh\e_N[u,u]	\over \|u\|^2_{\LL(Y)}},
\end{gather}
where $\mathfrak{V}[k]$  is the family of all $k$-dimensional subspaces in $\dom(\hh\e_N)$. We introduce the operator $P:\C_\Omega^{m+1}\to \LL(Y)$ by
$$
P\uu=\suml_{j=0}^{m}u_j\chi_{\Omega_j},
$$
where $\chi_{\Omega_j}$ is the indicator function of $\Omega_j$; since the $\Omega_j$'s are disjoint by assumption, we have
\begin{gather}
\label{unitar}
\|P\uu\|_{\LL(Y)}=\|\uu\|_{\C_\Omega^{m+1}},\quad \hh\e_N[P\uu,P\uu]=\hh_N[\uu,\uu].
\end{gather}
Let $\uu_{0,N},\dots,\uu_{m,N}$ be an orthonormal system of eigenvectors of $\HH_N$ such that $\HH_N u_{j,N}=\lambda_{j,N}u_{j,N}$. We denote $W_k:=\mathrm{span}(u_{0,N},\dots, u_{m,N})$, then it is easy to check that
\begin{gather}
\label{max}
\forall \uu\in W_k:\quad
{\hh_N[\uu,\uu]\over \|\uu\|^2_{\C_\Omega^{m+1}}}\leq \lambda_{k,N},
\end{gather}
the equality in \eqref{max} being attained   for $\uu=\uu_{k,N}$.

Finally, we set $V_k:=P W_k$. It is obvious that $V_k\in \mathfrak{V}[k]$ and using \eqref{minmax}-\eqref{max} we obtain
\begin{gather*}
\lambda_{k,N}\e\leq
\max_{u\in V_k\setminus\{0\}}{\hh\e_N[u,u]	\over \|u\|^2_{\LL(Y)}}=
\max_{\uu\in W_k\setminus\{0\}}{\hh_N[\uu,\uu]	\over \|\uu\|^2_{\C_\Omega^{m+1}}}=
\lambda_{k,N},
\end{gather*}
which concludes the proof.	
\end{proof}	

\begin{lemma}
	\label{lemmaN2}
For any $k\in\{1,\dots,m+1\}$ one has
\begin{gather}\label{estN2}
\lambda_{k,N}\leq \lambda_{k,N}\e + C\eps
\end{gather}
provided $\eps$ is small enough.	
\end{lemma}
\begin{proof}
For $u\in\dom(\hh_N\e)$	we introduce the norm
$$
\|u\|_{1,\eps}:=\left(\hh\e_N[u,u]+\|u\|^2_{\LL(Y)}\right)^{1/2}.
$$
Furthermore, we define the operator $\Phi:\dom(\hh_N\e)\to \dom(\hh_N)$ by
\begin{gather*}	
(\Phi u)_j={1\over |\Omega_j|}\int_{\Omega_j}u(x)\dd x,\quad j=0,\dots, m.
\end{gather*}	
Our goal is to prove that the following estimates hold for each $u\in\dom(\hh\e_N)$,
\begin{gather}
\label{estim1}
\|u\|_{\LL(Y)}^2\leq \|\Phi u\|_{\C_\Omega^{m+1}}^2 + C_1\eps^2\|u\|^2_{1,\eps},\\
\hh_N[\Phi u,\Phi u]\leq \hh_N\e[u,u] + C_2\eps\|u\|^2_{1,\eps}.\label{estim2}
\end{gather}
Then by means of Lemma~\ref{lemma-EP} from Appendix we will get
\begin{gather}\label{EP-estim}
\lambda_{k,N}\e\leq \lambda_{k,N} + {\lambda_{k,N}\e(1+\lambda_{k,N}\e)C_1\eps^2+(1+\lambda_{k,N}\e)C_2\eps \over 1-(1+\lambda_{k,N}\e)C_1\eps^2},
\end{gather}
and since $\lambda_{k,N}\e\leq \lambda_{k,N}$ holds by Lemma~\ref{lemmaN1}, the sought estimate \eqref{estN2} would follow from \eqref{EP-estim}.

Estimate \eqref{estim1} is an easy consequence of the Poincar\'e inequality
\begin{gather*}
\label{poincare}
\forall j\in\{0,\dots,m\}:\quad\|u-(\Phi u)_j\|_{\LL(\Omega_j)}\leq
C\|\nabla u\|_{\LL(\Omega_j)}.
\end{gather*}
Indeed, we have
\begin{multline*}
\|u\|_{\LL(Y)}^2=\suml_{j=0}^m\|u\|_{\LL(\Omega_j)}^2=
 \|\Phi u\|^2_{\C_\Omega^{m+1}}+ \suml_{j=0}^m\|u-(\Phi u)_j\|_{\LL(\Omega_j)}^2\\ \leq
  \|\Phi u\|^2_{\C_\Omega^{m+1}}+ C_1\suml_{j=0}^m\|\nabla u\|^2_{\LL(\Omega_j)}\leq
  \|\Phi u\|^2_{\C_\Omega^{m+1}}+ C_1\eps^2\|u\|^2_{1,\eps}.
\end{multline*}
Let us next prove \eqref{estim2}. One has
\begin{gather*}
\hh_N[\Phi u,\Phi u]\leq \hh_N\e[u,u] + \suml_{j=1}^m q_j R_j[u,u],
\end{gather*}
where
$$
R_j[u,u]:=\|(\Phi u)_0-(\Phi u)_j\|^2_{\LL(\Gamma_j)}-
\|u\restr\gaj\ext-u\restr\gaj\inn\|^2_{\LL(\Gamma_j)},
$$
and these expressions can be estimated in the following way,
\begin{multline*}
|R_j[u,u]|\leq \left|\left\|(\Phi u)_0-(\Phi u)_j\right\|_{\LL(\Gamma_j)}-\left\|  u\restr\gaj\ext-u\restr\gaj\inn\right\|_{\LL(\Gamma_j)}\right|\\\times
\left(\left\|(\Phi u)_0-(\Phi u)_j\right\|_{\LL(\Gamma_j)}+\left\|  u\restr\gaj\ext-u\restr\gaj\inn\right\|_{\LL(\Gamma_j)}\right)\\\leq
\left(\left\|(\Phi u)_0  - u\restr\gaj\ext \right\|_{\LL(\Gamma_j)}+\left\| (\Phi u)_j -  u\restr\gaj\inn\right\|_{\LL(\Gamma_j)}\right)\\\times
\left(\|(\Phi u)_0\|_{\LL(\Gamma_j)}+\|(\Phi u)_j\|_{\LL(\Gamma_j)} + \|u\restr\gaj\ext\|_{\LL(\Gamma_j)}+\|u\restr\gaj\inn \|_{\LL(\Gamma_j)}\right).
\end{multline*}
Using the trace and the Poincar\'e inequalities we get
\begin{multline}\label{p1}
j\in\{1,\dots,m\}:\quad \left\|(\Phi u)_j  - u\restr\gaj\inn \right\|_{\LL(\Gamma_j)} \leq
C\sqrt{\left\|(\Phi u)_j  - u \right\|^2_{\LL(\Omega_j)}+
\left\|\nabla u \right\|_{\LL(\Omega_j)}^2}\\\leq
C_1 \left\|\nabla u \right\|_{\LL(\Omega_j)} \leq C_1\eps \|u\|_{1,\eps},
\end{multline}
and similarly,
\begin{gather}\label{p2}
\left\|(\Phi u)_0  - u\restr\gaj\ext \right\|_{\LL(\Gamma_j)} \leq
C \left\|\nabla u \right\|_{\LL(\Omega_0)} \leq C\eps \|u\| _{1,\eps}.
\end{gather}
Using further the trace and Cauchy-Schwarz inequalities one finds
\begin{multline}\label{p3}
\|(\Phi u)_0\|_{\LL(\Gamma_j)}+\|(\Phi u)_j\|_{\LL(\Gamma_j)} + \|u\restr\gaj\ext\|_{\LL(\Gamma_j)}+\|u\restr\gaj\inn \|_{\LL(\Gamma_j)}\\\leq
C\|u\|_{\W(Y\setminus\cup_{j=1}^m\Gamma_j)}\leq C\|u\|_{1,\eps}.
\end{multline}
Combining now \eqref{p1}-\eqref{p3} we obtain the needed estimate,
\begin{gather*}
|R_j[u,u]|\leq C\eps \|u\|_{1,\eps},
\end{gather*}
which implies the validity of \eqref{estim2} concluding thus the proof.
\end{proof}	

Finally, we notice that the matrix of the form \eqref{HHN} was investigated already in \cite{BK15} (using  different notations). It was demonstrated there that its eigenvalues are the roots of the function $\lambda F(\lambda)$, where $F(\lambda)$ is defined by \eqref{mu_eq}. Taking this into account we immediately obtain the following corollary from the last two lemmata.

\begin{corollary}\label{corN}
One has
\begin{gather*}
\lambda\e_{1,N}=0,\quad \lambda\e_{k,N}\leq B_{k-1}\text{\quad for }\;k\in\{2,\dots,m+1\}.
\end{gather*}
Moreover, for small enough $\eps$ there is also a lower bound,
\begin{gather*}
B_{k-1}-C\eps\leq \lambda\e_{k,N} \text{\quad for }\;k\in\{2,\dots,m+1\}.
\end{gather*}
\end{corollary}	

As we have already noticed above, the limit operator in the $\phi_0$-periodic situation has the same eigenvalues as the Neumann one. We have the following claim the proof of which repeats \emph{verbatim} the argument of Lemmata~\ref{lemmaN1} and \ref{lemmaN2}.

\begin{lemma}\label{lemmaPer}
One has
\begin{gather*}
\lambda\e_{1,\phi_0}=0,\quad \lambda\e_{k,\phi_0}\leq B_{k-1}\text{\quad for }\;k\in\{2,\dots,m+1\}.
\end{gather*}
Moreover, for small enough $\eps$ there is also a lower bound,	
\begin{gather*}
B_{k-1}-C\eps\leq \lambda\e_{k,\phi_0} \text{\quad for }\;k\in\{2,\dots,m+1\}.
\end{gather*}
\end{lemma}

\subsection{Asymptotic behavior of $\lambda_{k,D}\e$ and $\lambda_{k,\phi_\pi}\e$}

Keeping the boldface symbols from the previous section, we denote by $\C^{m}_\Omega$ the space of vectors $\uu=(0,u_1,\dots,u_m)\in \C^{m+1}$ equipped with the scalar product
\begin{gather*}
(\uu,\mathbf{v})_{\C^{m}_\Omega}=\suml_{j=1}^m u_j \overline{v_j} |\Omega_j|,
\end{gather*}
and introduce in this space the sesquilinear form $\hh_D$,
$$
\hh_D[\uu,\mathbf{v}]:=\suml_{j=1}^m q_j |\Gamma_j| u_j\overline{v_j}
$$
with $\dom(\hh_D)=\C^{m }_\Omega$. Let further $\HH_D$ be the operator in $\C^{m }_\Omega$ associated with this form. It is clear that $\HH_D$ acts as
\begin{gather*}
\HH_D \uu=\sum_{j=1}^m q_1 |\Gamma_1||\Omega_1|^{-1} u_j
\end{gather*}
and its eigenvalues are $A_1,\ A_2,\dots,\, A_m$.

\begin{lemma}\label{lemmaD}
One has
\begin{gather*}
\lambda\e_{k,D}\leq A_k\text{\quad for }\;k\in\{1,\dots,m \}.
\end{gather*}
Moreover, for small enough $\eps$ there is a lower bound,
\begin{gather*}
A_{k}-C\eps\leq \lambda\e_{k,D} \text{\quad for }\;k\in\{1,\dots,m\}.
\end{gather*}
\end{lemma}	
The proof of this lemma is again similar to the proof of Lemmata~\ref{lemmaN1} and \ref{lemmaN2}. The only essential difference here is that instead of the Poincar\'e inequality in $\Omega_0$ we use the Friedrichs inequality,
$$
\|u\|_{\LL(\Omega_0)}\leq \|\nabla u\|_{\LL(\Omega_0)},
$$
which is valid because functions from $\dom(\hh_D\e)$ have zero trace on $\partial Y$.

The analogous result is valid for eigenvalues in the $\phi_\pi$-periodic situation.

\begin{lemma}\label{lemmaAntiper}
One has
\begin{gather*}
\lambda\e_{k,\phi_\pi}\leq A_k\text{\quad for }\;k\in\{1,\dots,m \}.
\end{gather*}
Moreover, for small enough $\eps$ there is a gain a lower bound
\begin{gather*}
A_{k}-C\eps\leq \lambda\e_{k,\phi_\pi} \text{\quad for }\;k\in\{1,\dots,m\}.
\end{gather*}
\end{lemma}

This brings us to the conclusion. Combining Corollary~\ref{corN}, Lemmata~\ref{lemmaPer}--\ref{lemmaAntiper}, and equations \eqref{repres1}--\eqref{enclosure} we arrive at the claim of Theorem~\ref{th1}.

\section*{Acknowledgment}

The research was supported by Project No. 17-01706S of the Czech Science Foundation (GA\v{C}R) and by the Czech-Austrian grant 7AMBL7ATO22. A.K. is supported by the Austrian Science Fund (FWF) under Project No. M~2310-N32.

\section*{Appendix}

Here we recall a result from \cite{EP05}, which is a simple consequence of the min-max principle and serves to compare eigenvalues of two operators acting in different Hilbert spaces.

Let $H$ and $H'$ be two separable Hilbert spaces with the norms $\|\cdot\|$ and $\|\cdot\|'$. Let $\Hl$ and $\Hl'$ be non-negative self-adjoint operators in these spaces with purely discrete spectra, and $\h$ and $\h'$ the corresponding forms, respectively. We denote by $\{\lambda_k\}_{k\in\N}$ and
$\{\lambda_k'\}_{k\in\N}$ the corresponding sequences of eigenvalues, numbered in the ascending order and with account of their multiplicity. Finally, we set $\|u\|^2_{n}:=\|u\|^2+\|\Hl^{n/2} u\|$.

\begin{lemma}\label{lemma-EP} \cite{EP05}
Suppose that $\Phi:\dom(\h)\to\dom(\h')$ is a linear map such that for all $u\in \dom (\Hl^{\max\{n_1,n_2\}/2})$ one has
\begin{gather*}
\|u\|^2\leq \|\Phi u\|'^2 + \delta_1\|u\|_{n_1}^2,\\
\h'[\Phi u,\Phi u]\leq \h[u,u] + \delta_2\|u\|_{n_2}^2.
\end{gather*}
with some constants $n_1,n_2 \geq 0$ and $\delta_1,\delta_2\geq 0$. Then for each $k\in\N$ we have
\begin{gather}
\label{est-EP}
\lambda_k'\leq \lambda_k + {\lambda_k(1+\lambda_k^{n_1})\delta_1+(1+\lambda_k^{n_2})\delta_2\over 1-(1+\lambda_k^{n_1})\delta_1}
\end{gather}
 provided the denominator $1-(1+\lambda_k^{n_1})\delta_1$ is positive.
\end{lemma}

\begin{remark}
{\rm The above result was established in \cite{EP05} under the assumption that $\dim H=\dim H'=\infty$, however, it is easy to see from its proof that the result remains valid for $\dim H<\infty$ as well. In that case \eqref{est-EP} holds for $k\in\{1,\dots,\,\dim H\}$. This is the situation in the proof of Lemma~\ref{lemmaN2}, where we apply Lemma~\ref{lemma-EP} to $H=\C_\Omega^{n+1}$.}
\end{remark}


\EndPaper


\end{document}